\documentclass[10pt]{amsart}

\usepackage{amsmath, amsfonts}
\usepackage[all]{xy}
\usepackage{footnote}

\long\def\symbolfootnote[#1]#2{\begingroup%
\def\thefootnote{\fnsymbol{footnote}}\footnote[#1]{#2}\endgroup}
\theoremstyle{plain}
\newtheorem{Theorem}{{Theorem}}[section]
\newtheorem{Lemma}[Theorem]{Lemma}

\newtheorem{Corollary}[Theorem]{Corollary}
\newtheorem{Definition}[Theorem]{Definition}

\usepackage[pdftex]{graphicx,color}
\usepackage[pdftex]{hyperref}
\usepackage{amsmath, amsfonts}
\usepackage[all]{xy}
\usepackage{footnote}
\usepackage{graphicx} 
\usepackage{float}

\title{Nim on hypercubes}
\author{\textbf{Lindsay Erickson,} \\ Concordia College \\ lerick15@cord.edu \\
\textbf{Warren Shreve,} \\ North Dakota State Univeristy \\ warren.shreve@ndsu.edu}
\begin{document}

\maketitle

\section*{Abstract}

The ordinary game of Nim has a long history and is well-known in the area of combinatorial game theory.  The solution to the ordinary game of Nim has been known for many years and lends itself to numerous other solutions to combinatorial games.  Nim was extended to graphs by taking a fixed graph with a playing piece on a given vertex and assigning positive integer weight to the edges that correspond to a pile of stones in the ordinary game of Nim.  Players move alternately from the playing piece across incident edges, removing weight from edges as they move. This paper solves Nim on hypercubes in the unit weight case completely.  We briefly discuss the arbitrary weight case and its ties to known results.

\section{Background}

The graphs we will consider are finite and undirected with no multiple edges or loops.  We will often want to label the vertices and edges.  When we do, the edge between vertex $v_i$ and $v_j$ will be denoted $e_{ij}$.  Additional graph theory terminology, including path, vertex degree, and graph isomorphism, will be assumed as found in \cite{MR2107429}.  When we refer to the length of a cycle or path, we will call it even or odd by the number of edges the cycle or path contains.

\subsection{How to Play}

To play Nim on graphs, two players first agree on a finite, undirected, integrally weighted graph and a fixed starting position.  The position of the game is indicated by a positional piece which we will denote by $\Delta$. The game starts with $P_1$ choosing an edge incident with $\Delta$ to move across.  As a player moves across an edge, the player must lower the weight of the edge by a positive integer amount.  The positional piece $\Delta$ moves with the move of the player so that when a player comes to rest on the other vertex incident with that edge, the next player must start with that vertex and move across edges incident with the new position of $\Delta$.  If either player lowers the weight of an edge to zero, the edge is no longer playable.  For ease of notation, we will delete the edge from the picture of the graph if the weight is decreased to zero (see Figure~\ref{fig:exgame}).  Play continues in this back-and-forth fashion until a player can no longer move since there are no edges incident with $\Delta$.

\begin{figure}[h]
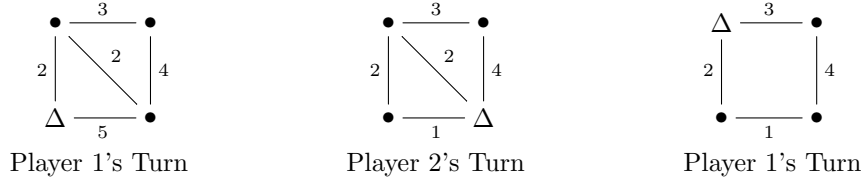


\caption{An example of the first two moves in a game of Nim on graphs.} \label{fig:exgame}

\begin{center}
$
\xygraph{ 
!{(0,0) }*+{\Delta}="a" 
!{(0,1) }*+{\bullet}="b" 
!{(1,1) }*+{\bullet}="c" 
!{(1,0)}*+{\bullet}="d"
"a"-"b"^{2} 
"a"-"d"_{5}
"b"-"c"^{3}
"b"-"d"^{2}
"c"-"d"^{4}
!{(3.5,0) }*+{\bullet}="e" 
!{(3.5,1) }*+{\bullet}="f" 
!{(4.5,1) }*+{\bullet}="g" 
!{(4.5,0)}*+{\Delta}="h"
"e"-"f"^{2} 
"e"-"h"_{1}
"f"-"g"^{3}
"f"-"h"^{2}
"g"-"h"^{4}
!{(7,0) }*+{\bullet}="i" 
!{(7,1) }*+{\Delta}="j" 
!{(8,1) }*+{\bullet}="k" 
!{(8,0)}*+{\bullet}="l"
"i"-"j"^{2} 
"i"-"l"_{1}
"j"-"k"^{3}
"k"-"l"^{4}
} 
$
$\begin{array}{ccccccccccccc}\mbox{Player 1's Turn} &   &  &  &  &  & \mbox{Player 2's Turn} &  &    & &  &  & \mbox{Player 1's Turn}\end{array}$
\end{center}

\end{figure}

\subsection{Nim on Graphs Defintions}

\begin{Definition}
Given a graph $G$ with edge set $E(G)$ and vertex set $V(G)$ we will call the non-negative integer value assigned to each $e \in E(G)$ the \textbf{weight} of the edge and denote the weight of edge $e_{ij}$ by $\omega(e_{ij})$. 
\end{Definition}

When we say a graph has \textit{unit weight}, we precisely mean that $\omega(e) = 1$ for all $e \in E(G)$.  We will often discuss \textit{uniformly weighted} graphs, meaning that $\omega(e) = k$ for all $e \in E(G)$ and for some $k \in \mathbb{Z}^+$.

For any graph $G$ we assume  $\omega(e_{ij}) \neq 0$ for all $e_{ij} \in E(G)$ at the start of a game.  When an edge is such that $\omega(e)=0$ we will delete it from the graph entirely, since it is no longer a playable edge.  Given a game graph $G$ with weight assignment $\omega_G (e)$, denote by $P_1$ the first player to move from the starting vertex, and denote by $P_2$ the player to move after $P_1$.  The indicator piece $\Delta$ denotes the vertex from which a player is to move.  We will always enumerate vertices in such a way that $\Delta$ is on $v_1$ at the start of a game.

\begin{Definition}
For either player and from a given position $\Delta$ on vertex $v_j$, we define the set of vertices to which a player may legally move to from $\Delta$ to be the \textbf{option} of the player.  The set of options of player $i$ at vertex $v_j$ will be denoted by $O(P_i, v_j)$.
\end{Definition}

Certainly for a vertex to exist in the set of options the incident edge must be adjacent to $\Delta$.  Thus $O(P_i, v_j)=\{v_k \in V(G) : \Delta = v_j; \ e_{jk} \in E(G);  \ \omega(e_{jk}) \neq 0\}$.  We will omit $v_j$ when the position of $\Delta$ is apparent.

\begin{Definition}
For either player and from a given position $\Delta$ on vertex $v_j$, we call the decision of how much weight to remove from an edge $e_{ji}$ the \textbf{choice} of the player.
\end{Definition}

Thus for any given option with $\omega(e_{ij}) > 1$, the player has a choice of whether or not to remove all weight, or exactly how much weight to remove.

\begin{Definition}
We will say that a pair of $P_i$'s options are \textbf{isomorphic} if given two options, $v_j, v_k \in O(P_i, v_i)$, there exists a graph isomorphism between $v_j$ and its neighbors and $v_k$ and its neighbors.  We will say that two options are \textbf{identical} if in addition to being isomorphic, the subgraph induced by $v_i$ and each $v_j \in O(P_i, v_i)$ have the same weight assignment.
\end{Definition}

We will use the word \textit{option} exclusively when we are referring to the vertex a player will move to, and the word \textit{choice} to refer to the amount of weight across the option's edge to be removed during play.  Hence during any given move, a player will have the option of which vertex to move to, and the choice of how much weight to remove.

Notice that the definition of isomorphic requires that the vertices in the set of options have the same degree, and that there is a bijection between the options of the vertices within the set of isomorphic options (see Figure~\ref{fig:isomorphicoptions}).  In other words, if for all $v_j, v_k \in O(P_i, v_i)$ we have that $O(P_i, v_j) \cong O(P_i,v_k)$ then the options of $v_i$ are isomorphic.  We will also talk about graphs being isomorphic within the context of Nim on graphs.  This will be necessary to cut down on cases to consider within games.

\vspace{.25in}

\begin{figure}[h]
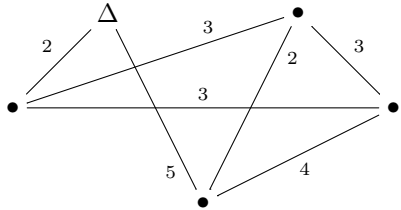


\caption{The options at $\Delta$ are isomorphic but not identical.} \label{fig:isomorphicoptions}

\begin{center}
$
\xygraph{ 
!{(-1,2) }*+{\Delta}="a" 
!{(1,2) }*+{\bullet}="b" 
!{(2,1) }*+{\bullet}="c" 
!{(0,0)}*+{\bullet}="d"
!{(-2,1)}*+{\bullet}="e" 
"a"-"e"_{2} 
"a"-"d"_(.8){5}
"b"-"c"^{3}
"b"-"d"^(.2){2}
"b"-"e"_(.3){3}
"c"-"d"^{4}
"c"-"e"_{3}
}
$
\end{center}

\end{figure}

\section{Nim on the hypercube with unit weight}

\begin{Definition}
The $n$-dimensional hypercube, or the $n$-cube, $Q_n$ is the graph $K_2$ if $n=1$, while for $n \geq 2$, $Q_n$ is defined recursively as $Q_{n-1} \times K_2$ \cite{MR2107429}.
\end{Definition}

We can also think of the $n$-cube as the graph whose vertices are labeled by the binary $n$-tuples $(a_1, a_2, \ldots, a_n)$ where each $a_i$ is either 0 or 1 for $1 \leq i \leq n$ and such that two vertices are adjacent if and only if their corresponding $n$-tuples differ at precisely one coordinate.  This is the view of hypercubes that we will adopt in what follows, along with the following alternate labeling.  Label each vertex $a = (a_1, a_2, \ldots, a_n)$ of the hypercube $Q_n$ by the corresponding set $X_a = \{i : a_i = 1\}$ \cite{MR1367739}.  Then we will draw the $Q_n$ in the plane so that the vertical coordinates of the vertices are in order by the size of the sets labeling them (see Figure~\ref{ex:Q_3}).  We will call this the level labeling scheme and use it throughout the hypercube section.

\begin{figure}[h]
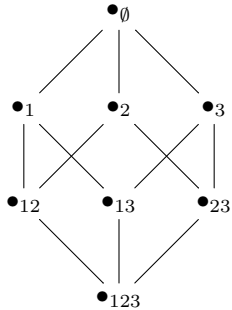

\caption{Here is the $Q_3$ with the level labeling scheme.}\label{ex:Q_3}
\begin{center}
$
\xygraph{
!{(0,0) }*+{\bullet_{123}}="123"
!{(-1,1) }*+{\bullet_{12}}="12"
!{(0,1) }*+{\bullet_{13}}="13"
!{(1,1) }*+{\bullet_{23}}="23"
!{(-1,2) }*+{\bullet_{1}}="1"
!{(0,2) }*+{\bullet_{2}}="2"
!{(1,2) }*+{\bullet_{3}}="3"
!{(0,3) }*+{\bullet_{\emptyset}}="\emptyset"
"123"-"12"
"123"-"13"
"123"-"23"
"12"-"1"
"12"-"2"
"13"-"1"
"13"-"3"
"23"-"2"
"23"-"3"
"1"-"\emptyset"
"2"-"\emptyset"
"3"-"\emptyset"
}
$
\end{center}

\end{figure}

\begin{Definition}

The parity of a vertex in $Q_n$ is the parity of the number of 1's in its name, even or odd \cite{MR1367739}.

\end{Definition}

This implies that each edge of the $Q_n$ has an even vertex and an odd vertex as endpoints (see Figure~\ref{ex:Q_3}).  This means that the even vertices form an independent set, as do the odd vertices.  Hence $Q_n$ is bipartite for any $n$ \cite{MR1367739}.

Since we typically start with $\Delta$ on the lowest numerically denoted vertex.  Here we will start with $\Delta$ on vertex $\emptyset$.  With this level labeling scheme, we can think of the vertices at different levels corresponding to the number of digits in the vertex labels.  Thus in the example of the $Q_3$ we have levels $\emptyset, 1, 2,$ and $3$.


Throughout this section we will assume that the weight of each edge of the hypercube has unit weight.

\begin{Lemma}\label{lem:levels012}
$P_1$ can keep game play on the $Q_{2n+1}$ within the confines of levels $\emptyset, 1,$ and $2$.
\end{Lemma}

\begin{proof}
Let $Q_{2n+1}$ have unit weight and label each vertex by the $X_a$ scheme described above so that $\emptyset$ is at vertex $(0,0, \ldots, 0)$.  Give the $Q_{2n+1}$ the level labeling scheme.  Assume that $\Delta$ starts at vertex $\emptyset$.  Note that since the $Q_{2n+1}$ is regular of order $2n+1$ any choice of starting vertex is isomorphic.

Every hypercube is bipartite.  Thus we can observe that $P_1$'s vertices all have even parity, and $P_2$'s vertices have odd parity according to the labeling scheme.

Suppose $P_1$ is at vertex $ij$ in level 2.  Since we want to show that $P_1$ can opt not to move down to level 3, we will show that there is always an option in level 1 for any $ij$ in level 2.  Since $P_1$ is playing from vertex $ij$, either $P_2$ moved from $i$ or from $j$ in level 1.  Without loss of generality, assume $P_2$ moved from $i$ so that $e_{i,ij}$ is no longer an option for $P_1$.

By way of contradiction, suppose that $P_1$ cannot move to $j$ from $ij$.  This implies that $e_{j,ij}$ has been used already.  This can only occur in one of two ways: the first case is that $P_1$ moved to $j$ via $e_{j,ij}$ on a previous move, and the second case is that P$_2$ moved from $j$ to $ij$ via $e_{j,ij}$ on a previous move.  

In the first case, if $P_1$ moved from $ij$ to $j$ then it must be the case that $P_2$ was on level three and moved from some $ijk$ to $ij$.  This is because we are assuming that just now $P_2$ moved from $i$ to $ij$ and thus could not have made that move previously. (Recall that since we have unit weight, once an edge has been moved across once it is no longer a playable edge.)  This contradicts the fact that $P_1$ would not make such a move unless forced to.  Clearly $P_1$ was not forced to previously move down to level 3 since it is only now that a move to vertex $i$ is no longer possible.

In the second case, if $P_2$ moved from $j$ to $ij$ but $P_1$ did not move from $ij$ to $i$ since it remains, then $P_1$ moved down to some $ijk$, again a contradiction.

Thus $P_1$ always has a level 1 option and hence can keep $P_2$ within levels $\emptyset, 1,$ and $2$.
\end{proof}


\begin{Theorem}\label{Thm:P1winsQ_{2n+1}}
Assume $\omega(e) = 1$ for all $e \in Q_{2n+1}$.  Then $P_1$ can win the $Q_{2n+1}$ for any $n \geq 1$.
\end{Theorem}

\begin{proof}
Assume $\omega(e) = 1$ for all $e \in E(Q_{2n+1})$, that $n \in \mathbb{Z}$, and $n \geq 1$.  Label the digits according to $X_a$ and the level labeling scheme.  Start with $\Delta$ on $\emptyset$.

Since all hypercubes are bipartite, we know that $P_1$'s vertices have even parity, and $P_2$ vertices have odd parity.  By Lemma~\ref{lem:levels012}, $P_1$ can keep $P_2$ within the confines of levels $\emptyset, 1,$ and $2$.  Because of this, consider only these three levels.  In essence, ``chop off" levels 3 through $2n+1$.

With $P_1$ at $\emptyset$ at the start, notice that the vertices in level $\emptyset$ and 1 are all odd degree.  Since we are considering the graph without levels 3 through $2n+1$, the vertices in level 2 are all of degree 2.  Also, since we are assuming each edge has unit weight, when a player moves across an edge, it is deleted.  Thus $P_1$ starts on an odd degree vertex and $P_2$ starts on an even degree vertex at each of their respective moves.  This implies that $P_1$ always has an edge to move away from at any vertex (since odd degree implies at least degree 1).  However, since $Q_{2n+1}$ is finite, eventually $P_2$ will come to a vertex of degree 0 and not be able to move.

Thus $P_1$ always wins the $Q_{2n+1}$ for any positive integer value of $n$.
\end{proof}

\begin{Theorem}

Assume that $\omega(e)=1$.  Then $P_2$ wins the $Q_{2n}$ for all $n \geq 1$.

\end{Theorem}

\begin{proof}
Assume that $n \in \mathbb{Z}$, $n \geq 1$, and $\omega(e) = 1$ for all $e \in E(Q_{2n})$.  Label the vertices according to $X_a$ and the level labeling scheme.  Start with $\Delta$ on $\emptyset$.

Note that $Q_{2n}$ is regular of degree 2n, and $Q_{2n}$ is bipartite.  Thus $P_1$ moves from vertices with even parity, and $P_2$ moves from vertices with odd parity.  Also notice that $P_1$ starts from a vertex of even degree, and each time $P_1$ moves from $\emptyset$ it is of even degree.  Each other vertex is of odd degree when either player moves from it.  This is because the degree lowers by one each time a player arrives at the vertex.  Thus on the first move, $P_1$ moves from an even degree vertex to what was an even degree vertex.  Since the process of moving to a vertex lowers the degree by one each time because of unit weight of the edges, $P_2$ starts from a vertex that has odd degree.  This is true for each player at each vertex except for $P_1$ at vertex $\emptyset$.


If a vertex has odd degree when moving from it, a player is guaranteed to be able to move away from the vertex, since an odd degree vertex implies that the degree is at least 1.  Thus the only vertex that a player could possibly get stuck at is the $\emptyset$ vertex.  Since $P_1$ is the only player to move from $\emptyset$ by virtue of $Q_{2n}$ being bipartite, $P_1$ is the only player who is able to lose.

Therefore, since there are only a finite number of moves, $P_2$ wins the $Q_{2n}$ for any positive integer value of $n$.
\end{proof}

With the previous two theorems, we can formulate the following two corollaries.

\begin{Corollary}
$P_1$ wins the unit weight hypercube if and only if $n$ is odd.
\end{Corollary}

\begin{Corollary}
$P_2$ wins the unit weight hypercube if and only if $n$ is even.
\end{Corollary}

\section{A note about the hypercube with arbitrary weight}

The unit weight hypercube had a nice parity argument to show the winner.  Unfortunately, the hypercube weighted arbitrarily is not so easy to solve.  We know very quickly that weight matters with the arbitrarily weighted hypercube.  Take for a simple example, $Q_2 = C_4$.  By previous work in the even cycle section, we know that the winner of the game is decided by the distances to the lowest weight edge.  Hence we can tell at least for the even values of $n$ that the weight of the $Q_n$ will matter in determining the winner of the game.

\bibliographystyle{plain}
\bibliography{AllReferences}

\end{document}